\documentclass[11pt,a4paper,bibliography=totocnumbered,listof=totocnumbered]{article}
\usepackage[ngerman,english]{babel}
\usepackage[T1]{fontenc}
\usepackage[utf8]{inputenc}
\usepackage{amsmath}
\usepackage{mathtools}
\usepackage{amsthm}
\usepackage{amsfonts}
\usepackage{amssymb}
\usepackage{graphicx} 
\usepackage{geometry} 
\usepackage{setspace} 
\usepackage{subfig} 

\usepackage{paralist} 
\usepackage{enumitem}

\usepackage[font=footnotesize,labelfont=bf]{caption}

\usepackage{titlesec} 
\usepackage[colorlinks, linkcolor = blue, citecolor = blue, urlcolor = blue]{hyperref}
\usepackage[toc,page]{appendix}
\usepackage{cite}
\usepackage{bbm}
\usepackage{dsfont}

\newcommand{\mE}{\mathcal{E}}

\DeclareMathOperator{\supp}{supp}

\newtheoremstyle{normal}
{10pt}
{10pt}
{}
{}
{\bfseries}
{}
{0em}
{\bfseries{\thmname{#1}\thmnumber{ #2}\thmnote{ \hspace{0em}(#3)\newline}}}

\newtheoremstyle{standard}  
  {10pt}   
  {}   
  {\itshape}  
  {}       
  {\bfseries} 
  {:}         
  {0.2cm}  
  {}          

\newtheoremstyle{mittitel}  
  {10pt}   
  {}   
  {\itshape}  
  {}       
  {\bfseries} 
  {:}         
  {0.2cm}  
  {\bfseries{\thmname{#1}\thmnumber{ #2}\thmnote{ \hspace{0em}(#3)\newline}}}          

\title{An Approximation of Solutions to Heat Equations defined by Generalized Measure Theoretic Laplacians \vspace{-1ex} }
\author{Tim Ehnes\footnote{ Institute of Stochastics and Applications, University of Stuttgart, Pfaffenwaldring 57, 70569 Stuttgart, Germany, Email: tim.ehnes@mathematik.uni-stuttgart.de. }~ and Ben Hambly\footnote{ Mathematical Institute, University of Oxford, Woodstock Road, Oxford, OX2 6GG, UK, Email: hambly@
maths.ox.ac.uk. }}
\date{\vspace{-5ex}}

\begin{document}
\maketitle

\setlength{\parindent}{10pt}

\titlespacing{\section}{0pt}{25pt plus 4pt minus 2pt}{18pt plus 2pt minus 12pt}
\titlespacing{\subsection}{0pt}{25pt plus 4pt minus 2pt}{14pt plus 2pt minus 2pt}
\titlespacing{\subsubsection}{0pt}{12pt plus 4pt minus 2pt}{-6pt plus 2pt minus 2pt}

\theoremstyle{standard}

\newtheorem{thm}{Theorem}[section] 
\newtheorem{satz}[thm]{Proposition} 
\newtheorem{lem}[thm]{Lemma}
\newtheorem{kor}[thm]{Corollary} 
\newtheorem{defi}[thm]{Definition} 
\newtheorem{bem}[thm]{Remark}
\newtheorem{hyp}[thm]{Assumption}
\newtheorem{exa}[thm]{Example}

\begin{abstract}
We consider the heat equation defined by a generalized measure theoretic Laplacian on $[0,1]$. This equation describes heat diffusion in a bar 
such that the mass distribution of the bar is given by a non-atomic Borel probabiliy measure $\mu$, where we do not assume the existence of a 
strictly positive mass density. We show that weak measure convergence implies convergence of the corresponding generalized Laplacians in the 
strong resolvent sense. We prove that strong semigroup convergence with respect to the uniform norm follows, which implies uniform convergence 
of solutions to the corresponding heat equations. This provides, for example, an interpretation for the mathematical model of heat diffusion 
on a bar with gaps in that the solution to the corresponding heat equation behaves approximately like the heat flow on a bar with sufficiently 
small mass on these gaps.
\end{abstract}

\section{Introduction}\label{Introduction}

Let $[a,b]\subset \mathbb{R}$ be a finite interval, $\mu$ be a non-atomic Borel probability measure on $[a,b]$ such that $a,b\in\supp(\mu)$, $\mathcal{L}^2([a,b],\mu)$ be the space of measurable functions $f$ such that $\int_a^b f^2d\mu<\infty$ and $L^2([a,b],\mu)$ be the corresponding Hilbert space of equivalence classes with inner product $\langle f,g\rangle_{\mu}\coloneqq \int_a^b fgd\mu$. We define
\begin{align*}
\mathcal{D}_{\mu}^2\coloneqq \Big\lbrace f\in C^1([a,b]):~ &\exists \left(f^{\prime}\right)^{\mu}\in L^2([a,b],\mu):~f^{\prime}(x)=f^{\prime}(a)+\int_a^x \left(f^{\prime}\right)^{\mu}(y)d\mu(y), ~~ x\in[a,b]\Big\rbrace.
\end{align*}
The Krein-Feller operator with respect to $\mu$ is given as 
\begin{align*}
\Delta_{\mu}: \mathcal{D}_{\mu}^2\subseteq L^2([a,b],\mu)\to L^2([a,b],\mu),~~ f\mapsto\left(f^{\prime}\right)^{\mu}.
\end{align*}
This definition involves the derivative with respect to $\mu$. If a function $f$ has a representation given by
\begin{align*}
f(x)=\int_a^x \frac{d }{d\mu}f(x)d\mu(x), ~ x\in[0,1],
\end{align*}
then $\frac{d }{d\mu}f$ is called the $\mu$-derivative of $f$.
Consequently, in the above definition, $\left(f^{\prime}\right)^{\mu}$ is the $\mu$-derivative of $f'$. \smallskip

This operator has been widely studied, for example with an emphasis on addressing questions of the spectral asymptotics and further analytical properties \cite{BNF,CNE,FA,FD,FAF,FAS,FG,FLZ,FME,FP,FR,FS,FZH,FRO,MS2,MS,NS,NTS}, diffusion processes \cite{IKD,KO,KS}, wave equations \cite{CJI} and higher-dimensional generalizations  \cite{FSD,NST,SVO}. \par

In order to connect these operators with diffusion equations from a physical point of view, we follow for example \cite[Section 1.2]{HE} and consider a 
metallic rod of constant cross-sectional area oriented in the $x$-direction occupying a region from $x=0$ to $x=1$ such that all thermal quantities 
are constant across a section. We can thus consider the rod as one-dimensional. We investigate the conduction of heat energy on a segment from $x=a$ to $x=b$. Let the temperature at the point $x\in[a,b]$ and time $t\in[0,\infty)$ denoted by $u(t,x)$ and the total thermal energy in the considered segment at time $t$ by $e_{a,b}(t)$. It is well-known that
\begin{align*}
e_{a,b}(t)=\int_a^b u(t,x)\rho(x)dx,
\end{align*}
assuming that the rod possesses a mass density $\rho:[0,1]\to(0,\infty)$. However, if we denote the mass distribution of the rod by $\mu$, we can write
\begin{align*}
e_{a,b}(t)=\int_a^b u(t,x)d\mu(x).
\end{align*}
Hence, we can define the total heat energy even if $\mu$ has no density. The total heat energy changes only if heat energy flows through the boundaries $x=a$ and $x=b$. We deduce for the rate of change of heat energy
\begin{align}\label{physic_intro_1}
\frac{d}{dt} e_{a,b}(t)=\phi(t,a)-\phi(t,b),
\end{align}
where $\phi(t,x)$ denotes the heat flux density, which gives the rate of thermal energy flowing through $x$ at time $t$ to the right. 
Assuming sufficient regularity, we can rewrite \eqref{physic_intro_1} as
\begin{align*}
\int_a^b \frac{\partial}{\partial t} u(t,x)d\mu(x) = -\int_a^b \frac{d}{d\mu}\phi_t(x)d\mu(x),
\end{align*}
where $\phi_t(x)\coloneqq \phi(t,x)$ and the $\mu$-derivative was defined earlier.
With $u_t(x)\coloneqq u(x,t)$, Fourier's law of heat conduction $\phi=-\frac{\partial u}{\partial x}$ gives  
\begin{align*}
\int_a^b \frac{\partial}{\partial t} u(t,x)d\mu(x) = \int_a^b \frac{d}{d\mu}\frac{d}{d x} u_t(x)d\mu(x).
\end{align*}
Since this is valid for all $a,b\in[0,1],$ $a<b$, it follows for $t\in[0,\infty)$ and $\mu$-almost all $x\in[0,1]$
\begin{align*}
\frac{\partial}{\partial t} u(t,x) =  \frac{d}{d\mu}\frac{d}{d x} u_t(x).
\end{align*}
Applying the definition of the Krein-Feller operator leads to the generalized heat equation
\begin{align}
\frac{\partial u}{\partial t} = \Delta_{\mu} u_t, ~~ t\in[0,\infty) \label{heat_equation_intro}
\end{align}
with Dirichlet boundary conditions $u(t,0)=u(t,1)=0$ for all $t\geq 0$ if we assume that the temperature vanishes at the boundaries or with Neumann boundary conditions 
$\frac{\partial u}{\partial x}(t,0)=\frac{\partial u}{\partial x}(t,1)=0$ if the boundaries are assumed to be perfectly insulated.
This provides a physical motivation for a mass distribution having full support even if it possesses no Lebesgue density. However, it is still not clear how to interpret the equation if the support of the mass distribution is not the whole interval, in particular for singular measures, such as measures on the Cantor set. \smallskip

The problem then is to describe heat flow on a rod with massless parts. Krein-Feller operators defined by measures on the classic Cantor set or, more general, 
Cantor-like sets with gaps have been extensively studied in recent years (see e.g. \cite{AM,FLZ,FME,FS,FR}). In this paper, we give an interpretation of a solution to \eqref{heat_equation_intro} 
in the case where $\mu$ is not of full support. We approximate the solution by a sequence of solutions to heat equations defined by $\mu_n$ for $n\in\mathbb{N}$ such that $\mu_n$ is of full support and converges weakly to $\mu$ for $n\to\infty$. \smallskip

To this end, let $b\in\{N,D\}$ represent the boundary condition, where $N$ denotes Neumann and $D$  Dirichlet boundary conditions and we give our basic assumption.
\begin{hyp}\label{ass:1}
Let $\left(\mu_n\right)_{n\in\mathbb{N}}$ be a sequence of non-atomic Borel probability measures on $[0,1]$ such that $0,1\in \supp(\mu_n)$ and $\mu_n \rightharpoonup \mu, n\to\infty$, where $\rightharpoonup$ denotes weak measure convergence.
\end{hyp}
It is well-known that $\Delta_{\mu}^b$ is a non-positive self-adjoint operator (see, e.g.,  \cite{FA}) and thus the generator of a strongly continuous semigroup $\left(T_t^b\right)_{t\geq 0}$ (see, e.g. \cite[Lemma 1.3.2]{FOD}).
If $u_0\in L^2([0,1],\mu)$, then the unique solution to the initial value problem 
\begin{align}
\begin{split}
\frac{\partial u}{\partial t}(t) &= \Delta_{\mu}^b u(t), ~~ t\in[0,\infty),\\
u(0)&=u_0
\end{split}
\end{align}
is given by $u(t)=T_t^b u_0$, according to a generalized solution concept we introduce later.
This motivates the investigation of strong semigroup convergence. However, for different measures, the corresponding semigroups are defined on different spaces. For the special case $\supp(\mu)=\supp(\mu_n)=[0,1]$ for all $n\in\mathbb{N}$, the results in \cite{CS} can be applied to obtain strong semigroup convergence on the space of continuous functions on $[0,1]$. 
 To formulate a strong semigroup convergence result without that assumption, we restrict the semigroup $\left( T_t^N\right)_{t\geq 0}$ associated to $\Delta_{\mu}^N$ on $L^2([0,1],\mu)$ to the subspace of continuous functions, denoted by $C([0,1])_{\mu}^N$, which is a Banach space with the uniform norm. The semigroup $\left( T_t^D\right)_{t\geq 0}$ is restricted to the Banach space of continuous functions satisfying Dirichlet boundary conditions, denoted by $C([0,1])_{\mu}^D$.
  We show that the restricted semigroup, which we denote by $\left( \bar T_t^b\right)_{t\geq 0}$, is, again, a strongly continuous contraction semigroup 
and the infinitesimal generator is given by
\begin{align*}
\bar \Delta_{\mu}^b f \coloneqq \Delta_{\mu}^b f, ~~ \mathcal{D}\left(\bar \Delta_{\mu}^b\right)\coloneqq \left\{ f\in\mathcal{D}\left(\Delta_{\mu}^b\right): \Delta_{\mu}^bf\in C[0,1])_{\mu}^b\right\}.
\end{align*} 
Moreover, if we assume that $\supp(\mu)\subseteq \supp(\mu_n)$, the space $C([0,1])_{\mu}^b$ can be continuously embedded in $C([0,1])_{\mu_n}^b$, where we denote the embedding by $\pi_n$. We will see that in this case, strong semigroup convergence is equivalent to strong resolvent convergence and strong resolvent convergence is what we will establish.
  More precisely, let $f\in C([0,1])_{\mu}^b,$ $\lambda>0$ and $n\in\mathbb{N}$. We define $\bar R_{\lambda}^b\coloneqq \left(\lambda-\bar \Delta_{\mu}^b\right)^{-1} $ and $\bar R_{\lambda,n}^b\coloneqq \left(\lambda-\bar \Delta_{\mu_n}^b\right)^{-1} $ and prove
\begin{align}\label{eq:resolventconv}
\left\lVert \pi_n \bar R_{\lambda}^bf-\bar R_{\lambda,n}^b\pi_n f\right\rVert_{\infty}
\to 0, ~~ n\to\infty .
\end{align}
The main tool for proving \eqref{eq:resolventconv} is the generalization of the hyperbolic functions $\sinh$ and $\cosh$, defined by generalizing the series
\begin{align*}
\sinh(zx) = \sum_{k=0}^{\infty}z^{2k+1}\frac{x^{2k+1}}{(2k+1)!}, ~~
\cosh(zx) = \sum_{k=0}^{\infty}z^{2k}\frac{x^{2k}}{(2k)!}.
\end{align*}
We replace $\frac{x^k}{k!}$ by generalized monomials defined by a measure $\mu$. This extends the theory of measure theoretic functions, developed for trigonometric functions in \cite{AM}. Then, we show that the resolvent density of the operator $\Delta_{\mu}^b$ is a product of such generalized hyperbolic functions. This leads to the desired strong resolvent convergence by proving convergence of these generalized hyperbolic functions.
We obtain

\begin{thm}\label{strong_semigroup_con_theorem}
Let $f\in (C[0,1])_{\mu}^b$ and $\mu_n$ be a sequence of measures satisfying Assumption~\ref{ass:1}. Then, for all $t\geq 0$
\begin{align*}
\lim_{n\to\infty} \left\lVert \pi_n \bar T_t^b f - \bar T_{t,n}^b\pi_n f\right\rVert_{\infty} = 0,
\end{align*}
uniformly on bounded time intervals.
\end{thm}
After that, we will see that for $f\in(C[0,1])_{\mu}^b$ $\left\{u(t)=\bar T_t^bf: t\geq 0\right\}$ is the unique solution to the initial value problem 
\begin{align}\label{heat_equation_intro_2}
\begin{split}
\frac{\partial u}{\partial t}(t) &= \bar \Delta_{\mu}^b u(t), ~~ t\in[0,\infty), \\
u(0)&=f
\end{split}
\end{align}
in the sense that $t\mapsto u(t)$ satisfies \eqref{heat_equation_intro_2} for all $t>0$ and is continuous with respect to $(C[0,1])_{\mu}^b$ for all $t\geq 0$. Analogously, 
$\left\{u_n(t)=\bar T_{t,n}^bf: t\geq 0\right\}$ is the unique solution to the initial value problem 
\begin{align*}
\begin{split}
\frac{\partial u_n}{\partial t}(t) &= \bar \Delta_{\mu_n}^b u_n(t), ~~ t\in[0,\infty), \\
u_n(0)&=\pi_n f.
\end{split}
\end{align*}

Finally,  combining these results and Theorem 
\ref{strong_semigroup_con_theorem} yields
\begin{align*}
\lim_{n\to\infty} \left\lVert \pi_n u(t)- u_n(t)\right\rVert_{\infty}=0,
\end{align*}
uniformly on bounded time intervals. 
 
We obtain a meaningful interpretation for the diffusion of heat in the case of a mass distribution with gaps in that the heat in a rod with mass 
distribution $\mu$ diffuses approximately like the heat on a rod with mass distribution $\mu_n$ for sufficiently large $n$.  
\\

This paper is structured as follows. In the following section, we recall definitions related to Krein-Feller operators. In Section \ref{Generalized Hyperbolic Functions and the Resolvent Operator}, we introduce the concept of generalized hyperbolic functions and the connection to resolvent operators. Section \ref{The Restricted Semigroup} is devoted to the restriction of the Krein-Feller operator semigroup to the spaces $(C[0,1])_{\mu}^b$ for $b\in\{N,D\}$. After these preparations, in Section \ref{Convergence results} we develop the central convergence results, namely the convergence of the hyperbolic functions and the strong resolvent convergence in Section \ref{Strong Resolvent Convergence}, the graph norm convergence of the considered operators in Section \ref{Graph Norm Convergence} and finally, the strong semigroup convergence and convergence of solutions to heat equations in Section \ref{Strong Semigroup Convergence}. In Section \ref{Application} we show how to apply the results in three examples. Lastly, in Section \ref{Directions for Further Research}, we discuss some open problems.

\section{Preliminaries} \label{Preliminaries}
 
First, we recall the definition and some analytical properties of the operator $\Delta_{\mu}^b$, where $b\in\{N,D\}$ and $\mu$ is a non-atomic Borel probability measure on $[0,1]$ such that $0,1\in\supp(\mu)$.
 If $[0,1]\setminus \supp(\mu)\neq\emptyset$, then $[0,1]\setminus \supp(\mu)$ is open in $\mathbb{R}$ and can be written as
\begin{align}
[0,1]\setminus \supp(\mu) = \bigcup_{i\geq 1} (a_i,b_i)\label{offene_mengen} 
\end{align}
with $0<a_i<b_i<1$, $a_i,b_i\in\supp(\mu)$ for $i\geq 1$. We define 
\begin{align*}
\mathcal{D}^1\coloneqq\left\lbrace f:[0,1]\to\mathbb{R}: \text{there exists } f^{\prime}\in L^2\left([0,1],\lambda^1\right): f(x)=f(0)+\int_0^x f^{\prime}(y)dy,~ x\in[0,1]  \right\rbrace
\end{align*}
and $H^1\left([0,1],\lambda^1\right)$ to be the space of all $L^2([0,1],\mu)$-equivalence classes possessing a $\mathcal{D}^{1}-$repre-sentative. If $\mu=\lambda^1$ on $[0,1]$, this definition is equivalent to the definition of the Sobolev space $W_2^1$.

We observe that $H^1\left([0,1],\lambda^1\right)$ is the domain of the non-negative symmetric bilinear form $\mE$ on $L^2([0,1],\mu)$ defined by
\begin{align*}
\mathcal{E}(u,v)=\int_0^1 u'(x)v'(x)dx,~~~ u,v\in \mathcal{F}\coloneqq H^1\left([0,1],\lambda^1\right).
\end{align*}
It is known (see \cite[Theorem 4.1]{FD}) that $\left(\mE,\mathcal{F}\right)$ defines a Dirichlet form on $L^2([0,1],\mu)$. Hence, there exists an associated non-negative, self-adjoint operator $\Delta_{\mu}^N$ on $L^2([0,1],\mu)$ with $\mathcal{F}=\mathcal{D}\left(\left(-\Delta_{\mu}^N\right)^{\frac{1}{2}}\right)$ such that
\begin{align*}
\langle -\Delta_{\mu}^N u,v\rangle_{\mu}&=\mE(u,v), ~~u\in\mathcal{D}\left(\Delta_{\mu}^N\right),v\in \mathcal{F}
\end{align*}
and 
\[\mathcal{D}\left(\Delta_{\mu}^N\right)=\left\{f\in L^2([0,1],\mu): f \text{ has a representative } \bar f \text{ with } \bar f\in \mathcal{D}_{\mu}^2 \text{ and } \bar f'(0)=\bar f'(1)=0\right\}.\] 
The operator $\Delta_{\mu}^N$ is called the Neumann Krein-Feller operator with respect to $\mu$. Furthermore, let $\mathcal{F}_0$ be the space of all $L^2([0,1],\mu)$-equivalence classes having a $\mathcal{D}^1-$representative $f$ such that $f(0)=f(1)=0.$
The bilinear form defined by
\begin{align*}
\mathcal{E}(u,v)=\int_0^1 u'(x)v'(x)dx,~~~ u,v\in \mathcal{F}_0,
\end{align*}
is a Dirichlet form, too (see \cite[Theorem 4.1]{FD}).
Again, there exists an associated non-negative, self-adjoint operator $\Delta_{\mu}^D$ on $L^2([0,1],\mu)$ with $\mathcal{F}_0=\mathcal{D}\left(\left(-\Delta_{\mu}^D\right)^{\frac{1}{2}}\right)$ such that
\begin{align*}
\langle -\Delta_{\mu}^D u,v\rangle_{\mu}&=\mE(u,v), ~~ u\in \left(\Delta_{\mu}^D\right),~v\in \mathcal{F}_0
\end{align*}
and
\[\mathcal{D}\left(\Delta_{\mu}^D\right)=\left\{f\in L^2([0,1],\mu): f \text{ has a representative } \bar f \text{ with } \bar f\in \mathcal{D}_{\mu}^2 \text{ and } \bar f(0)=\bar f(1)=0\right\}.\] Then $\Delta_{\mu}^D$ is called the Dirichlet Krein-Feller operator with respect to $\mu$.\\

Furthermore, it is known from \cite[Proposition 6.3, Lemma 6.7, Corollary 6.9]{FA} that there exists an $L_2([0,1],\mu)$-orthonormal basis $\lbrace \varphi_k^b: k\in\mathbb{N}\rbrace$  consisting of eigenfunctions of $-\Delta_{\mu}^b$ and that for the related ascending ordered eigenvalues $\{\lambda^{b}_i:i\in\mathbb{N}\}$ we have $0\leq\lambda_1^b\leq\lambda_2^b\leq...,$ where $\lambda_1^D>0$.

\section{Generalized Hyperbolic Functions and the Resolvent Operator}\label{Generalized Hyperbolic Functions and the Resolvent Operator}

Let $b\in\{N,D\}$ and let $\mu$ be defined as before. In this section we develop a useful representation for the resolvent density of $\Delta_{\mu}^b$. \par 
Let $\lambda>0$. We consider the initial value problem
\begin{align}
\begin{cases}	
\Delta_{\mu} g = \lambda g, \\
g(0)=1, ~~g^{\prime}(0)=0 
\end{cases}\label{ivpN1}
\end{align}
on $L^2([0,1],\mu)$. The problem \eqref{ivpN1} possesses a unique solution (see \cite[Lemma 5.1]{FA}), which we denote by $g_{1,N}^{\lambda}$. Further, under the initial conditions 
\begin{align}
g(1)=1, ~~g^{\prime}(1)=0, \label{ivpN2} \\ 
g(0)=0, ~~g^{\prime}(0)=1,\label{ivpD1}
\end{align}
and 
\begin{align}
g(1)=0, ~~g^{\prime}(1)=1 ,  \label{ivpD2}
\end{align}
respectively, the above eigenvalue problem also possesses a unique solution (see\cite[Remark 5.2]{FA}), and we denote it by $g_{2,N}^{\lambda}$, $g_{1,D}^{\lambda}$ and $g_{2,D}^{\lambda}$, respectively. The resolvent density is then given as follows.
\begin{lem}\cite[Theorem 6.1]{FA}
Let $\lambda>0$. The resolvent operator $R_{\lambda}^b\coloneqq (\lambda-\Delta_{\mu}^b)^{-1}$ is well-defined and for all $f\in L^2([0,1],\mu)$ we have
\begin{align*}
R_{\lambda}^{b}f(x)=\int_0^1\rho_{\lambda}^b(x,y)f(y)d\mu(y), ~~~ x\in[0,1],
\end{align*}
where the resolvent densities are given by
\begin{align*}
\rho_{\lambda}^N(x,y) = \rho_{\lambda}^N(y,x) &\coloneqq \frac{g_{1,N}^{\lambda}(x)g_{2,N}^{\lambda}(y)}{\left(g_{1,N}^{\lambda}\right)^{\prime}(1)} , ~~ x,y\in[0,1],~ x\leq y, \\
\rho_{\lambda}^D(x,y) = \rho_{\lambda}^D(y,x) &\coloneqq  -\frac{g_{1,D}^{\lambda}(x)g_{2,D}^{\lambda}(y)}{g_{1,D}^{\lambda}(1)} , ~~ x,y\in[0,1],~ x\leq y.
\end{align*}
\end{lem}
It is well-known that if $\mu=\lambda^1$, the solutions to \eqref{ivpN1} and  \eqref{ivpD1} are given by \[g_{1.N}^{\lambda}(x)=\cosh\left(\sqrt{\lambda}x\right) ~\text{ and }~  g_{1.D}^{\lambda}(x)=\frac{1}{\sqrt{\lambda}}\sinh\left(\sqrt{\lambda}x\right), ~x\in[0,1],\]
respectively.
We generalize the notion of hyperbolic functions by solving \eqref{ivpN1} and  \eqref{ivpD1} for an arbitrary measure $\mu$ according to the given conditions. 
To this end, we introduce generalized monomials as in \cite{AM}.

\begin{defi}\label{monomial_def}
For $x\in[0,1]$ we set $p_0(x)=q_0(x)=1$ and for $k\in\mathbb{N}$
\begin{align*}
p_k(x)&\coloneqq\begin{cases}
\int_0^xp_{k-1}(t)d\mu(t),& \text{if } k \text{ is odd,}\\
\int_0^xp_{k-1}(t)dt,& \text{if } k \text{ is even,}\
\end{cases}\\
q_k(x)&\coloneqq\begin{cases}
\int_0^xq_{k-1}(t)dt,& \text{if } k \text{ is odd,}\\
\int_0^xq_{k-1}(t)d\mu(t),& \text{if } k \text{ is even.}\
\end{cases}
\end{align*}
\end{defi}

We note that for $x\in[0,1]$ and $k\geq 0$,
\begin{align}\label{trig_estimate1}
p_{2k+1}(x)\leq  p_{2k}(x)\leq \frac{x^k}{k!},~~ q_{2k+1}(x)\leq p_{2k}(x)\leq\frac{x^k}{k!}
\end{align}
(see \cite[Lemma 2.3]{FLZ}).

\begin{defi}\label{trigonometric_def}
\label{trigonometric_def_iii} We define for $x\in[0,1]$, $z\in\mathbb{R}$
\begin{align*}
\sinh_{z}(x)\coloneqq \sum_{k=0}^{\infty}z^{2k+1}q_{2k+1}(x), ~~
\cosh_z(x)\coloneqq \sum_{k=0}^{\infty}z^{2k}p_{2k}(x).
\end{align*}
\end{defi}
By \eqref{trig_estimate1} for all $z\in\mathbb{R}$
\begin{align}
\left\lVert \sinh_z\right\rVert_{\infty}\leq ze^{z^2}, ~ 
\left\lVert \cosh_z\right\rVert_{\infty}\leq e^{z^2}.\label{cosh_sup_estimate}
\end{align}
\begin{exa} If $\mu=\lambda^1$, we have
$q_k(x)=\frac{x^k}{k!}, ~ k\geq 0$.
It follows that in this case
\begin{align*}
\sinh_z(x)=\sum_{k=0}^{\infty}z^{2k+1}\frac{x^{2k+1}}{(2k+1)!}=\sinh(zx)
\end{align*}
and analogously $\cosh_z(x)=\cosh(zx)$.
\end{exa}

\begin{satz}
Let $\lambda>0$. Then, for $x\in[0,1]$, we have
\begin{align*}
\begin{aligned}
g_{1,N}^{\lambda}(x) = &\cosh_{\sqrt{\lambda}}(x), ~~~
&g_{1,D}^{\lambda}(x) = &\frac{1}{\sqrt{\lambda}}\sinh_{\sqrt{\lambda}}(x),\\
g_{2,N}^{\lambda}(x) = &\cosh_{\sqrt{\lambda}}(1-x), ~~~
&g_{2,D}^{\lambda}(x) = &-\frac{1}{\sqrt{\lambda}}\sinh_{\sqrt{\lambda}}(1-x).
\end{aligned}
\end{align*}
\end{satz}
\begin{proof}

The assertion for $g_{1,D}^{\lambda}$ was proven in \cite[Lemma 2.3]{FLZ}. The proof for $g_{1,N}^{\lambda}$ works analogously. We verify the assertion for $g_{2,N}^{\lambda}$. Let $x\in[0,1]$. Then,
\begin{align*}
\cosh_{\sqrt{\lambda}}(1-x) &= \sum_{n=0}^{\infty}\lambda^n p_{2n}(1-x) \\
&=1+\sum_{n=1}^{\infty} \lambda^n \int_0^{1-x}\int_0^y  p_{2n-2}(t)d\mu(t)dy \\
&=1+\sum_{n=1}^{\infty} \lambda^n \int_0^{1-x}\int_{1-y}^{1}  p_{2n-2}(1-t)d\mu(t)dy \\
&=1-\sum_{n=1}^{\infty} \lambda^n \int_x^1\int_0^y  p_{2n-2}(1-t)d\mu(t)dy \\
&=1-\sum_{n=0}^{\infty} \lambda^{n+1} \int_x^1\int_0^y  p_{2n}(1-t)d\mu(t)dy .
\end{align*}
Due to estimate \eqref{trig_estimate1} we can use the dominated convergence theorem and obtain
\begin{align*}
\cosh_{\sqrt{\lambda}}(1-x) 
&=1- \lambda \int_x^1\int_0^y  \sum_{n=0}^{\infty}\lambda^n p_{2n}(1-t)d\mu(t)dy \\
&=1- \lambda \int_x^1\int_0^y  \cosh_{\sqrt{\lambda}}(1-t)d\mu(t)dy. 
\end{align*}
We set $f(x)\coloneqq \cosh_{\sqrt{\lambda}}(1-x),~ x\in[0,1]$ and get
\begin{align*}
f(x)=1- \lambda \int_x^1\int_0^y  f(t)d\mu(t)dy, ~ x\in[0,1]
\end{align*}
and in particular
\begin{align*}
f(0)=1- \lambda \int_0^1\int_0^y  f(t)d\mu(t)dy.
\end{align*}
It follows that, for $x\in[0,1]$,
\begin{align*}
f(x)-f(0)= \lambda \int_0^x\int_0^y  f(t)d\mu(t)dy.
\end{align*}
The latter equation can be written as $\Delta_{\mu}f=\lambda f.$ It remains to verify the initial conditions. Obviously, $f(1)=\cosh_{\sqrt{\lambda}}(0)=1$. 
Using \eqref{trig_estimate1} again, we have
\begin{align*}
f^{\prime}(1)=-\sum_{n=1}^{\infty}\lambda^n p_{2n-1}(0)=0.
\end{align*} 
The proof for $g_{2,D}^{\lambda}$ follows using the same ideas.
\end{proof}

This leads to the following representation for the resolvent density:
\begin{kor}
Let $\lambda>0$. It holds for $x,y\in[0,1]$, $x\leq y$,
\begin{align*}
\rho_{\lambda}^N(x,y)=\rho_{\lambda}^N(y,x)&= \left(\cosh^{\prime}_{\sqrt{\lambda}}(1)\right)^{-1}\cosh_{\sqrt{\lambda}}(x)\cosh_{\sqrt{\lambda}}(1-y),\\
\rho_{\lambda}^D(x,y)= \rho_{\lambda}^D(y,x)&= \frac{1}{\sqrt{\lambda}} \left(\sinh_{\sqrt{\lambda}}(1)\right)^{-1}\sinh_{\sqrt{\lambda}}(x)\sinh_{\sqrt{\lambda}}(1-y).
\end{align*}
\end{kor}

\section{The Restricted Semigroup}\label{The Restricted Semigroup}

Let $b\in\{N,D\}$ and let $\mu$ be defined as before. It is well-known that $\Delta_{\mu}^b$ is the generator of a strongly continuous Markovian  semigroup $\left( T_t^b\right)_{t\geq 0}$ of contractions on $L^2([0,1],\mu)$. 
\begin{defi}
For $(t,x,y)\in(0,\infty)\times[0,1]\times[0,1]$, we define
\begin{align*}
p_t^b(x,y)\coloneqq \sum_{k=1}^{\infty} e^{-\lambda_k^b t}\varphi_k^b(x)\varphi_k^b(y).
\end{align*}
This is called the heat kernel of $\Delta_{\mu}^b$. 
\end{defi}
The heat kernel is the integral kernel of the semigroup $\left( T_t^b\right)_{t\geq 0}$. That is,
for $t>0$ and $f\in L^2([0,1],\mu)$, we can write
\begin{align*}
T_t^bf(x)=\int_0^1 p^b_t(x,y)f(y)d\mu(y), ~~ x\in[0,1].
\end{align*}
In this section, we restrict these semigroups to appropriate spaces of equivalence classes of continuous functions. 
\begin{defi}
\begin{enumerate}[label=(\roman*)] 
\item  We  define $(C[0,1])_{\mu}^N$ as the set of all $L^2([0,1],\mu)$-equivalence classes possessing a continuous representative, formally 
\begin{align*}
(C[0,1])_{\mu}^N\coloneqq \left\{ f\in L^2([0,1],\mu): f \textnormal{ possesses a continuous representative}\right\}.
\end{align*}
\item We  further define $(C[0,1])_{\mu}^D$ as the set of all $L^2([0,1],\mu)$-equivalence classes possessing a continuous representative that satisfies Dirichlet boundary conditions, formally
\begin{align*}
(C[0,1])_{\mu}^D\coloneqq \big\{ f\in L^2([0,1],\mu): f \textnormal{ possesses a continuous representative } \bar f\\ \textnormal{ such that } \bar f(0)=\bar f(1)=0 \big\}.
\end{align*}
\end{enumerate}

\end{defi}
The space $(C[0,1])_{\mu}^b$ is a Banach space with the norm $\left\lVert f\right\rVert_{(C[0,1])_{\mu}^b}\coloneqq \left\lVert \left.f\right|_{\supp(\mu)} \right\rVert_{\infty}$.  Note that \[\left\lVert f\right\rVert_{(C[0,1])_{\mu}^b}=\left\lVert \widetilde f\right\rVert_{\infty},\] where $\widetilde f$ is the continuous representative of $f$ that is linear on all intervals in $[0,1]\setminus \supp(\mu)$. To simplify the notation, we henceforth write $\left\lVert f\right\rVert_{\infty}$ for $\left\lVert f\right\rVert_{(C[0,1])_{\mu}^b}$.\par 
Let $u=\sum_{k\geq 1} u_k^b\varphi_k^b\in L^2([0,1],\mu)$ and let $t>0$. It holds 
\begin{align}\label{semigroup_range}
\Delta_{\mu}^b T_t^b u = \sum_{k\geq 1}\lambda_k^be^{-\lambda_k^b t}u_k^b\varphi_k^b\in L^2([0,1],\mu)
\end{align}
and thus $T^b_t u\in\mathcal{D}\left(\Delta_{\mu}^b\right)$.
Hence, the following inclusion holds:
 \[T_t^b\left((C[0,1])_{\mu}^b\right)\subseteq (C[0,1])_{\mu}^b. \]  This motivates the definition of the restricted semigroup $
\left(\bar T_t^b\right)_{t\geq 0}\coloneqq \left(\left(T_t^b\right)_{|_{(C[0,1])_{\mu}^b}}\right)_{t\geq 0}
$, which is for $t\geq 0$ defined by
\begin{align*}
\bar T_t^b: (C[0,1])_{\mu}^b\to (C[0,1])_{\mu}^b, ~ \bar T_t^b f = T_t^bf.
\end{align*}
When evaluating an element of $(C[0,1])_{\mu}^b$ pointwise, we always evaluate the representative that is linear on all intervals in $[0,1]\setminus \supp(\mu)$.

The goal of this section is to show that $\left(\bar T_t^b\right)_{t\geq 0}$ again defines a strongly continuous contraction semigroup. It is obvious that the semigroup property holds. 
Note that by the Markov property of $(T_t^b)_{t\geq 0}$ 
for $g\in (C[0,1])_{\mu}^N$
\begin{align*}
\left| T_t^b g(x)\right|= \left| \int_0^1 p^b_t(x,y)g(y)d\mu(y)\right| \leq \left\lVert  g \right\rVert_{\infty} \left| \int_0^1 p^b_t(x,y)d\mu(y)\right|\leq  \left\lVert  g \right\rVert_{\infty}, ~ x\in[0,1].
\end{align*}
Hence, $(\bar T_t^b)_{t\geq 0}$ is a semigroup of contractions. It remains to prove the strong continuity. To this end, we need some preparations. 
We write $\mE(f,f)\coloneqq \mE(f)$ and $\|f\|_{\mu} = \int_0^1 f(x)^2 d\mu(x)$.

\begin{lem} \label{strong_cont_lem_1}
There exists a constant $c_1>0$ such that for all $f\in\mathcal{F}$
\begin{align*}
\left\lVert f\right\rVert_{\infty} \leq c_1\left(\mE(f)^{\frac{1}{2}}+\left\lVert f\right\rVert_{\mu}\right).
\end{align*}
\end{lem} 
\begin{proof}
We follow the proof of \cite[Lemma 5.2.8]{KA}. Let $f\in\mathcal{F}.$ Then, by the Cauchy-Schwarz inequality for all $x,y\in[0,1]$ 
\begin{align*}
\left| f(x)-f(y)\right| = \left| \int_x^y f^{\prime}(z)dz\right| \leq \left(\int_x^y\left( f^{\prime}\right)^2(z)dz\right)^{\frac{1}{2}}\left| x-y\right|^{\frac{1}{2}} = \mE(f)^{\frac{1}{2}}\left| x-y\right|^{\frac{1}{2}}.
\end{align*}
Now, let $g\in\mathcal{F}_0$. Then, by setting $y=0$ in the previous calculation, we get
\begin{align*}
\left|g(x)\right| \leq \mE(g)^{\frac{1}{2}}|x|^{\frac{1}{2}},~ x\in[0,1]
\end{align*}
and consequently,
\begin{align}\label{kig_lemma_1}
\left\lVert g\right\rVert_{\infty} \leq \mE(g)^{\frac{1}{2}}.
\end{align}
Further, for a given $f\in\mathcal{F}$, let $f_0$ be the unique harmonic function that coincides with $f$ on the boundary, that is $f_0(x)\coloneqq f(0)+x\left(f(1)-f(0)\right), ~x\in[0,1]$. Then,
\begin{align*}
\mE(f-f_0)&=\mE(f)-2\mE(f,f_0)+\mE(f_0)\\
&=\mE(f)-2\int_0^1f'(x)(f(1)-f(0))dx+(f(1)-f(0))^2\\
&=\mE(f)-2(f(1)-f(0))^2+(f(1)-f(0))^2\\
&=\mE(f)-(f(1)-f(0))^2
\end{align*}
and thus
\begin{align}\label{kig_lemma_2}
\mE(f-f_0)\leq\mE(f).
\end{align}
Combining \eqref{kig_lemma_1} and \eqref{kig_lemma_2},
\begin{align*}
\left\lVert f-f_0\right\rVert_{\infty}\leq \mE(f-f_0)^{\frac{1}{2}}\leq \mE(f)^{\frac{1}{2}}.
\end{align*}
Since the space of harmonic functions on $[0,1]$ is two-dimensional, there exists a constant $c_2>0$ such that for all $f\in\mathcal{F}$, the corresponding harmonic function $f_0$ satisfies
\begin{align*}
\left\lVert f_0\right\rVert_{\infty} \leq c_2\left\lVert f_0\right\rVert_{\mu}. 
\end{align*}
Combining the previous inequalities,
\begin{align*}
\left\lVert f\right\rVert_{\infty} &\leq \left\lVert f-f_0\right\rVert_{\infty}  + \left\lVert f_0\right\rVert_{\infty} \\
&\leq \mE(f)^{\frac{1}{2}} +  c_2\left\lVert f_0\right\rVert_{\mu}\\
&\leq \mE(f)^{\frac{1}{2}} +  c_2\left\lVert f-f_0\right\rVert_{\mu}+c_2\left\lVert f\right\rVert_{\mu}\\
&\leq \left(1+c_2\right)\mE(f)^{\frac{1}{2}}+c_2\left\lVert f\right\rVert_{\mu}.
\end{align*}
\end{proof}
\begin{lem}
Let $f\in (C[0,1])_{\mu}^b$. Then, $\lim_{t\to 0}\left\lVert T_t^b f-f\right\rVert_{\infty}=0.$\label{strong_cont_i}
\end{lem}
\begin{proof}
We follow the proof of \cite[Proposition 5.2.6]{KA}. Let $f\in\mathcal{F}$. By Lemma \ref{strong_cont_lem_1} and \cite[Lemma B.2.4]{KA},
\begin{align*}
\lim_{t\to 0}\left\lVert T_t^bf-f\right\rVert_{\infty}\leq c_1\left(\lim_{t\to 0} \mE\left(T_t^bf-f\right)+\left\lVert T_t^bf-f\right\rVert_{\mu}\right) = 0.
\end{align*}
By the fact that $\mathcal{F}$ is dense in $(C[0,1])_{\mu}^N$ and that, for $t\geq 0$, $T_t^N$ is continuous on $(C[0,1])_{\mu}^N$, we obtain the assertion for $b=N$.
To verify the case $b=D$, we prove that $\mathcal{F}_0$ is dense in $(C[0,1])_{\mu}^D$. Let $f\in (C[0,1])_{\mu}^D$. Then, by the density of $\mathcal{F}$ in $(C[0,1])_{\mu}^N$, there exists a sequence $\left( f_n\right)_{n\in\mathbb{N}}$ with $f_n\in \mathcal{F}$ for each $n\in\mathbb{N}$ such that 
\begin{align}\label{dense_proof}
\left\lVert f-f_n\right\rVert_{\infty}\to 0, ~ n\to\infty.
\end{align}
 We define for $n\in\mathbb{N}$
\begin{align*}
f_{n,0}(x) \coloneqq f_n(x)-f_n(0)-x(f_n(1)-f_n(0)), ~ x\in[0,1],
\end{align*}
which is an element of $\mathcal{F}_0$. Further, we have that
\begin{align*}
f_0(x)\coloneqq f(x)-f(0)-x(f(1)-f(0)) = f(x), ~ x\in[0,1],
\end{align*}
since $f$ satisfies Dirichlet boundary conditions.
This along with \eqref{dense_proof} implies for $n\in\mathbb{N}$
\begin{align*}
&\lim_{n\to\infty}\left\lVert f_{n,0}- f\right\rVert_{\infty} \\
&=\lim_{n\to\infty}\left\lVert f_{n,0}- f_0\right\rVert_{\infty}\\
&\leq \lim_{n\to\infty}\sup_{x\in[0,1]}\left| f_n(x)-f(x) \right| + \left| f_n(0)-f(0)\right| + \left| x\left(f_n(1)-f_n(0)-(f(1)-f(0))\right)\right|\\
&=0.
\end{align*}
\end{proof}

The main result of this section now follows immediately.
\begin{kor} \label{semigroup_cor}
$\left(\bar T_t^b\right)_{t\geq 0}$ is a strongly continuous contraction semigroup on $(C[0,1])_{\mu}^b$.
\end{kor}
\section{Convergence results}\label{Convergence results}

\subsection{Strong Resolvent Convergence} \label{Strong Resolvent Convergence}
 
Let $\mu$ be defined as before and let $F$ be the distribution function of $\mu$. Further, let $\left(\mu_n\right)_{n\in\mathbb{N}}$ satisfy Assumption~\ref{ass:1} and let $F_n$ be the distribution function of $\mu_n$ for $n\in\mathbb{N}$.

First, we give convergence results for the generalized hyperbolic functions introduced in Section \ref{Generalized Hyperbolic Functions and the Resolvent Operator} using results from \cite{FME}.
Let $p_k$, $q_k$, $k\in\mathbb{N}$ be defined by $\mu$ and $p_{k,n}, q_{k,n}$, $k\in\mathbb{N}$ be defined by $\mu_{n}$ for $n\in\mathbb{N}$.
\begin{lem}\cite[Lemma 3.1]{FME}
For $x\in[0,1]$ and $k, n\in\mathbb{N}$ we have
\begin{align*}
\begin{aligned}
|q_{2k}(x)-q_{2k,n}(x)| &\leq 2\frac{\, \left\lVert F-F_n\right\rVert_{\infty} x^k}{\, (k-1)!}  ,&|p_{2k}(x)-p_{2k,n}(x)| &\leq   2\frac{\, \left\lVert F-F_n\right\rVert_{\infty} x^k}{\, (k-1)!},\\
|q_{2k+1}(x)-q_{2k+1,n}(x)| &\le 2\frac{\, \left\lVert F-F_n\right\rVert_{\infty} x^k}{\, (k-1)!}, &|p_{2k+1}(x)-p_{2k+1,n}(x)| &\le  2\frac{\, \left\lVert F-F_n\right\rVert_{\infty} x^k}{\, (k-1)!}. 
\end{aligned}\end{align*}
\end{lem}


\begin{bem}
Since the distribution function of $\mu$ is continuous, weak measure convergence implies uniform convergence of the corresponding distribution functions (see \cite[Section 8.1]{CTP}), which is the condition in \cite[Lemma 3.1]{FME}.
\end{bem}
For $z\in\mathbb{R}$ let $\cosh_z, ~\sinh_z$ be defined by $\mu$ and $\cosh_{z,n}, ~\sinh_{z,n}$ be defined by $\mu_{n}$ for $n\in\mathbb{N}$. 
We obtain a result for the generalized hyperbolic functions, comparable to that for the trigonometric functions in \cite{FME}. 

\begin{lem} \label{cosh_estimate}
Let $z\in\mathbb{R}$. Then, 
\begin{align*}
\left\lVert \cosh_z-\cosh_{z,n}\right\rVert_{\infty} &\leq  2z^2 e^{ z^2} \left\lVert F-F_n\right\rVert_{\infty} , \\
\left\lVert \cosh^{\prime}_z-\cosh^{\prime}_{z,n}\right\rVert_{\infty} &\leq  \left(z^2+2z^4 e^{ z^2}\right) \left\lVert F-F_n\right\rVert_{\infty},\\
\left\lVert \sinh_z-\sinh_{z,n}\right\rVert_{\infty}  &\leq   2z^3 e^{ z^2} \left\lVert F-F_n\right\rVert_{\infty}.
\end{align*}
\end{lem}
\begin{proof}
Let $x \in [0,1]$ and $n\in\mathbb{N}$. Then,
\begin{align*}
\left|\cosh_z(x) - \cosh_{z,n} (x)\right| &\le  \sum_{k=1}^\infty |p_{2k}(x)-p_{2k,n}(x)|\, z^{2k} \\
&\le   \,  \sum_{k=1}^\infty  \frac{  2\left\lVert F-F_n\right\rVert_{\infty} }{\, (k-1)!}\, z^{2k} \\
&=  \,  \sum_{k=0}^\infty  \frac{ 2\left\lVert F-F_n\right\rVert_{\infty}}{\, k!}\, z^{2k+2} \\
&= \,  2z^2 e^{ z^2} \left\lVert F-F_n\right\rVert_{\infty}.
\end{align*}
Further, note that \[\cosh_z^{\prime}(x)=\sum_{k=1}^{\infty} p_{2k-1}(x)z^{2k}\] and \[\left|p_{1}(x)-p_{1,n}(x)\right|= \left|\mu([0,x])-\mu_n([0,x])\right| \leq \left\lVert F-F_n\right\rVert_{\infty}.\] With that,
\begin{align*}
\left|\cosh^{\prime}_z(x) - \cosh^{\prime}_{z,n} (x)\right| &\le  \sum_{k=1}^\infty |p_{2k-1}(x)-p_{2k-1,n}(x)|\, z^{2k} \\
&\le   \,  \left(z^2+2\sum_{k=2}^\infty  \frac{ z^{2k} }{\, (k-2)!}\,  \right)\left\lVert F-F_n\right\rVert_{\infty}\\
&\le \,  \left(z^2+2z^4 e^{ z^2}\right) \left\lVert F-F_n\right\rVert_{\infty}.
\end{align*}
Finally,
\begin{align*}
\left|\sinh_z(x) - \sinh_{z,n} (x)\right| &\le  \sum_{k=1}^\infty |q_{2k+1}(x)-q_{2k+1,n}(x)|\, z^{2k+1} \\
&\le   \,  \sum_{k=1}^\infty  \frac{  2\left\lVert F-F_n\right\rVert_{\infty}}{\, (k-1)!}\, z^{2k+1} \\
&\le   \,  \sum_{k=0}^\infty  \frac{  2\left\lVert F-F_n\right\rVert_{\infty}}{\, k!}\, z^{2k+3} \\
&\le \,  2z^3 e^{ z^2} \left\lVert F-F_n\right\rVert_{\infty},
\end{align*}
\end{proof}

We turn to the main result of this section. For $b\in\{N,D\}$ and $\lambda>0$, let $R_{\lambda}^b$ be defined by $\mu$ and  
$R_{\lambda,n}^b$ be defined by $\mu_{n}$. 
We assume $\supp(\mu)\subseteq \supp(\mu_n)$ for all $n\in\mathbb{N}$. Then, the mapping
\begin{align}
\pi_n: (C[0,1])_{\mu}^b\to  (C[0,1])_{\mu_n}^b,~  f\mapsto f\label{embedding}
\end{align}
 defines an embedding, where $f\in(C[0,1])_{\mu_n}^b$ denotes the $L^2([0,1],\mu_n)$-equivalence class of the representative of $f\in (C[0,1])_{\mu}^b$ that is linear on each interval $I\subseteq\supp(\mu_n)\setminus \supp(\mu)$.

\begin{thm} \label{resolvent_conv}
Let $\lambda>0$. Then, for all $f\in (C[0,1])_{\mu}^b$,
\begin{align*}
\lim_{n\to\infty}\left\lVert R_{\lambda,n}^b\pi_n f-\pi_n R_{\lambda}^bf\right\rVert_{\infty} = 0.
\end{align*}
\end{thm}
\begin{proof}
We simplify the notation in this proof by omitting all embeddings. If we evaluate on $\supp(\mu_n)\setminus \supp(\mu)$, we always evaluate the representative that is linear on each interval $I\subseteq\supp(\mu_n)\setminus \supp(\mu)$.
First, we consider the case $b=N.$ Let $\lambda>0$, $n\in\mathbb{N}$, $x,y\in[0,1]$ with $x\leq y$. Using the triangle inequality, 
\begin{align}
\begin{split}
&\left|\rho_{\lambda}^N(x,y)-\rho_{\lambda,n}^N(x,y)\right| \\ 
&\leq \left|\left(\cosh^{\prime}_{\sqrt{\lambda}}(1)\right)^{-1}-\left(\cosh^{\prime}_{\sqrt{\lambda},n}(1)\right)^{-1}\right|\left|\cosh_{\sqrt{\lambda}}(x)\cosh_{\sqrt{\lambda}}(1-y)\right|\\
&+\left|\cosh_{\sqrt{\lambda}}(x)-\cosh_{\sqrt{\lambda},n}(x)\right|\left|\left(\cosh^{\prime}_{\sqrt{\lambda},n}(1)\right)^{-1}\cosh_{\sqrt{\lambda}}(1-y)\right|\\
&+\left|\cosh_{\sqrt{\lambda}}(1-y)-\cosh_{\sqrt{\lambda},n}(1-y)\right|\left|\left(\cosh^{\prime}_{\sqrt{\lambda},n}(1)\right)^{-1}\cosh_{\sqrt{\lambda},n}(x)\right|. \label{resolvent_convergence_proof}
\end{split}
\end{align}
We have 
\begin{align}
\cosh^{\prime}_{\sqrt{\lambda}}(1)=\sum_{n=1}^{\infty}\lambda^n p_{2n-1}(1)\geq \lambda p_1(1) = \lambda \label{lambda_estimate}
\end{align}
 and similarly $\cosh^{\prime}_{\sqrt{\lambda},n}(1)\geq \lambda$. Applying this along with Lemma \ref{cosh_estimate}, we get
\begin{align*}
\left|\left(\cosh^{\prime}_{\sqrt{\lambda}}(1)\right)^{-1}-\left(\cosh^{\prime}_{\sqrt{\lambda},n}(1)\right)^{-1}\right| 
&= \left|\frac{\cosh^{\prime}_{\sqrt{\lambda},n}(1)-\cosh^{\prime}_{\sqrt{\lambda}}(1)}{\cosh^{\prime}_{\sqrt{\lambda}}(1)\cosh^{\prime}_{\sqrt{\lambda},n}(1)}
\right| \\
&\leq  \frac{ \left(\lambda+2\lambda^2 e^{\lambda}\right)\left\lVert F-F_n\right\rVert_{\infty}}{\lambda^2}
\end{align*}
and thus with  \eqref{cosh_sup_estimate}
\begin{align*}
\left|\left(\cosh^{\prime}_{\sqrt{\lambda}}(1)\right)^{-1}-\left(\cosh^{\prime}_{\sqrt{\lambda},n}(1)\right)^{-1}\right|\left|\cosh_{\sqrt{\lambda}}(x)\cosh_{\sqrt{\lambda}}(1-y)\right| \leq \frac{\left(e^{2\lambda}+ 2\lambda e^{3 \lambda}\right) \left\lVert F-F_n\right\rVert_{\infty}}{\lambda}.
\end{align*}
For the second term on the right-hand side of inequality \eqref{resolvent_convergence_proof}, we calculate
\begin{align*}
\left|\cosh_{\sqrt{\lambda}}(x)-\cosh_{\sqrt{\lambda},n}(x)\right|\left|\left(\cosh^{\prime}_{\sqrt{\lambda},n}(1)\right)^{-1}\cosh_{\sqrt{\lambda}}(1-y)\right| &\leq 
 2e^{2\lambda}\left\lVert F-F_n\right\rVert_{\infty}.
\end{align*}
Treating the third term analogously and using the above calculations in \eqref{resolvent_convergence_proof} yields
\begin{align*}
\lim_{n\to\infty}\max_{x\in[0,1]}\left|\rho_{\lambda}^N(x,y)-\rho_{\lambda,n}^N(x,y)\right| &\leq 
\lim_{n\to\infty} \frac{ \left(e^{2\lambda}+ 2\lambda e^{3 \lambda}\right) \left\lVert F-F_n\right\rVert_{\infty}}{\lambda} + 4e^{2\lambda}\left\lVert F-F_n\right\rVert_{\infty}\\
&=\lim_{n\to\infty}\left(\frac{1}{\lambda}+2e^{\lambda}+4\right)  e^{ 2\lambda} \left\lVert F-F_n\right\rVert_{\infty}\\ & = 0.
\end{align*}
Further, by \eqref{cosh_sup_estimate} and \eqref{lambda_estimate},
\begin{align*}
&\left|\int_0^1 \rho_{\lambda}^N(x,y)f(y)d\mu(y)- \int_0^1 \rho_{\lambda}^N(x,y)f(y)d\mu_n(y)\right| \\
&\leq \left|\left(\cosh^{\prime}_{\sqrt{\lambda}}(1)\right)^{-1}\cosh_{\sqrt{\lambda}}(x) \right|\left|\int_0^1 \cosh_{\sqrt{\lambda}}(1-y)f(y)d\mu(y) - \int_0^1 \cosh_{\sqrt{\lambda}}(1-y)f(y)d\mu_n(y)\right|\\
&\leq \frac{e^{\lambda}}{\lambda}\left|\int_0^1 \cosh_{\sqrt{\lambda}}(1-y)f(y)d\mu(y) - \int_0^1 \cosh_{\sqrt{\lambda}}(1-y)f(y)d\mu_n(y)\right|.
\end{align*}
Due to weak measure convergence,
\begin{align*}
\lim_{n\to\infty} \int_0^1 \cosh_{\sqrt{\lambda}}(1-y)f(y)d\mu_n(y) - \int_0^1 \cosh_{\sqrt{\lambda}}(1-y)f(y)d\mu(y) = 0
\end{align*}
and consequently,
\begin{align*}
\lim_{n\to\infty}\max_{x\in[0,1]} \left|\int_0^1 \rho_{\lambda}^N(x,y)f(y)d\mu(y)- \int_0^1 \rho_{\lambda}^N(x,y)f(y)d\mu_n(y)\right| = 0.
\end{align*}
We get the same result for $x\geq y$ and obtain
\begin{align*}
&\lim_{n\to\infty}\max_{x\in[0,1]}\left|R_{\lambda,n}^Nf(x)-R_{\lambda}^Nf(x)\right|\\ 
&\leq\lim_{n\to\infty}\max_{x\in[0,1]}\left|\int_0^1 \rho_{\lambda}^N(x,y)f(y)d\mu(y)- \int_0^1 \rho_{\lambda}^N(x,y)f(y)d\mu_n(y)\right| \\ &~~+\lim_{n\to\infty}\max_{x\in[0,1]}
\left|\int_0^1 \left(\rho_{\lambda}^N(x,y)-\rho_{\lambda,n}^N(x,y)\right) f(y)d\mu_n\right|\\
&=0.
\end{align*}
Now, let $b=D$. 
Again using the triangle inequality, for $n\in\mathbb{N}$, $x,y\in[0,1]$, $x\leq y$,
\begin{align}
\begin{split}
&\left|\rho_{\lambda}^D(x,y)-\rho_{\lambda,n}^D(x,y)\right| \\ 
&\leq \frac{1}{\sqrt{\lambda}}\Bigg( \left|\left(\sinh_{\sqrt{\lambda}}(1)\right)^{-1}-\left(\sinh_{\sqrt{\lambda},n}(1)\right)^{-1}\right|\left|\sinh_{\sqrt{\lambda}}(x)\sinh_{\sqrt{\lambda}}(1-y)\right|\\
&+\left|\sinh_{\sqrt{\lambda}}(x)-\sinh_{\sqrt{\lambda},n}(x)\right|\left|\left(\sinh_{\sqrt{\lambda},n}(1)\right)^{-1}\sinh_{\sqrt{\lambda}}(1-y)\right|\\
&+\left|\sinh_{\sqrt{\lambda}}(1-y)-\sinh_{\sqrt{\lambda},n}(1-y)\right|\left|\left(\sinh^{\prime}_{\sqrt{\lambda},n}(1)\right)^{-1}\sinh_{\sqrt{\lambda},n}(x)\right|\Bigg). \label{resolvent_convergence_proof_2}
\end{split}
\end{align}
We have 
\begin{align*}
\sinh_{\sqrt{\lambda}}(1)=\sum_{n=0}^{\infty}\lambda^{n+\frac{1}{2}} q_{2n+1}(1)\geq \sqrt{\lambda} q_1(1) = \sqrt{\lambda} 
\end{align*}
and thus 
\begin{align*}
\left|\left(\sinh_{\sqrt{\lambda}}(1)\right)^{-1}-\left(\sinh_{\sqrt{\lambda},n}(1)\right)^{-1}\right| 
&\leq   2\sqrt{\lambda} e^{\lambda}\left\lVert F-F_n\right\rVert_{\infty}.
\end{align*}
Arguing in the same way as before, we get
\begin{align*}
\lim_{n\to\infty}\max_{x\in[0,1]}\left|\rho_{\lambda}^D(x,y)-\rho_{\lambda,n}^D(x,y)\right| &\leq 
\lim_{n\to\infty}\frac{2}{\sqrt{\lambda}}\sqrt{\lambda}  e^{ \lambda} \left\lVert F-F_n\right\rVert_{\infty}\lambda e^{2\lambda}\\ &~~~+\lim_{n\to\infty} \frac{4}{\sqrt{\lambda}} \lambda^{\frac{3}{2}} e^{\lambda}\left\lVert F-F_n\right\rVert_{\infty} e^{\lambda}\\
&=\lim_{n\to\infty}\left(2e^{\lambda}+4\right)  \lambda e^{ 2\lambda}\left\lVert F-F_n\right\rVert_{\infty}\\
&=0.
\end{align*}
Further,
\begin{align*}
&\max_{x\in[0,1]}\left|\int_0^1 \rho_{\lambda}^D(x,y)f(y)d\mu(y)- \int_0^1 \rho_{\lambda}^D(x,y)f(y)d\mu_n(y)\right| \\
&\leq \max_{x\in[0,1]}\left|\left(\sqrt{\lambda}\sinh_{\sqrt{\lambda}}(1)\right)^{-1}\sinh_{\sqrt{\lambda}}(x) \right|\bigg|\int_0^1 \sinh_{\sqrt{\lambda}}(1-y)f(y)d\mu(y) \\ &\hspace{7cm}- \int_0^1 \sinh_{\sqrt{\lambda}}(1-y)f(y)d\mu_n(y)\bigg|\\
&\leq \left|\left(\sqrt{\lambda}\sinh_{\sqrt{\lambda}}(1)\right)^{-1}\right|\left\lVert\sinh_{\sqrt{\lambda}}\right\rVert_{\infty}\bigg|\int_0^1 \sinh_{\sqrt{\lambda}}(1-y)f(y)d\mu(y)\\&\hspace{7cm} - \int_0^1 \sinh_{\sqrt{\lambda}}(1-y)f(y)d\mu_n(y)\bigg|.\\
\end{align*}
Due to the weak measure convergence, this goes to zero as $n$ tends to $\infty$. 
Deducing the same result for $x\geq y$ and combining the above inequalities,
\begin{align*}
\lim_{n\to\infty}\max_{x\in[0,1]}&\left|R_{\lambda,n}^Df(x)-R_{\lambda}^Df(x)\right|\\ 
&\leq \lim_{n\to\infty}\max_{x\in[0,1]} \left|\int_0^1 \rho_{\lambda}^D(x,y)f(y)d\mu(y)- \int_0^1 \rho_{\lambda}^D(x,y)f(y)d\mu_n(y)\right| \\ &~~+\lim_{n\to\infty}\max_{x\in[0,1]}
\left|\int_0^1 \left(\rho_{\lambda}^D(x,y)-\rho_{\lambda,n}^D(x,y)\right)f(y)d\mu_n\right|\\
&=0.
\end{align*}
\end{proof}

\subsection{Graph Norm Convergence}\label{Graph Norm Convergence}

Let $\mu$ be defined as before
and let $\lambda>0$.  Analogously to the restricted semigroup, we define the restricted resolvent operator by
\begin{align*}
\bar R_{\lambda}^N: (C[0,1])_{\mu}^N\to (C[0,1])_{\mu}^N, ~ \bar R_{\lambda}^N f = R_{\lambda}^Nf, \\
\bar R_{\lambda}^D: (C[0,1])_{\mu}^D\to (C[0,1])_{\mu}^D, ~ \bar R_{\lambda}^D f = R_{\lambda}^Df.
\end{align*} 
Further, we define the operators $\bar \Delta_{\mu}^N$ and $\bar \Delta_{\mu}^D$ by
\begin{align*}
\bar \Delta_{\mu}^N f \coloneqq \Delta_{\mu}^N f, ~~  \mathcal{D}\left(\bar \Delta_{\mu}^N\right)\coloneqq \left\{ f\in\mathcal{D}\left(\Delta_{\mu}^N\right): \Delta_{\mu}^Nf\in (C[0,1])_{\mu}^N\right\},\\
\bar \Delta_{\mu}^D f \coloneqq \Delta_{\mu}^D f, ~~ \mathcal{D}\left(\bar \Delta_{\mu}^D\right)\coloneqq \left\{ f\in\mathcal{D}\left(\Delta_{\mu}^D\right): \Delta_{\mu}^Df\in (C[0,1])_{\mu}^D\right\},
\end{align*}
which are called the part of the operator $\Delta_{\mu}^N$ in $C[0,1])_{\mu}^N$ and the part of the operator $\Delta_{\mu}^D$ in $C[0,1])_{\mu}^D$, respectively.
The following Lemma shows how the restricted semigroup, the restricted resolvent and the part of the operator are connected. For that, let $b\in\{N,D\}$.
\begin{lem}
\begin{enumerate}[label=(\roman*)]
\item The infinitesimal generator of the strongly continuous contraction semigroup $\left(\bar T_t^b\right)_{t\geq 0}$ is $\bar  \Delta_{\mu}^b$. \label{restricted_lem_i}
\item $\bar R_{\lambda}^b$ is the resolvent of $\bar \Delta_{\mu}^b$.\label{restricted_lem_ii}
\end{enumerate}
\end{lem}

\begin{proof}
For all $f\in L^2([0,1],\mu)$, it holds $\left\lVert f\right\rVert_{\infty}\geq \left\lVert f\right\rVert_{\mu}$, therefore the inclusion map $i:(C[0,1])_{\mu}^b\to L^2([0,1],\mu), ~ f\mapsto f$ is continuous. Moreover, $\left(\bar T_t^b\right)_{t\geq 0}$ defines a strongly continuous contraction semigroup on $(C[0,1])_{\mu}^b$  and  $(C[0,1])_{\mu}^b$ is $\left(\bar T_t^b\right)_{t\geq 0}$-invariant (see Corollary \ref{semigroup_cor}). We thus can  apply \cite[II.2.3 Proposition]{ENO} to verify \ref{restricted_lem_i}. \smallskip  
We turn to part \ref{restricted_lem_ii}. Let $\lambda>0$ and let $\widetilde R_{\lambda}^b$ be the resolvent of $\bar \Delta_{\mu}^b$. 
 By part \ref{restricted_lem_i} and \cite[1.10 Theorem]{ENO}, this operator is well-defined and given by  
\begin{align*}
\widetilde R_{\lambda}^bf=\int_0^{\infty} e^{-\lambda s} \bar T_s^bfds, ~ f\in (C[0,1])_{\mu}^b.
\end{align*}
Further, by definition of $\left(\bar T_t^b\right)_{t\geq 0}$ and $\bar R_{\lambda}^b$,
\begin{align*}
\bar R_{\lambda}^bf = R_{\lambda}^bf =  
\int_0^{\infty} e^{-\lambda s} T_s^bfds = \int_0^{\infty} e^{-\lambda s} \bar T_s^bfds,  ~ f\in (C[0,1])_{\mu}^b.
\end{align*}
It follows $\widetilde R_{\lambda}^b =  \bar R_{\lambda}^b$ on $(C[0,1])_{\mu}^b$. 
\end{proof}

We are now able to establish graph norm convergence. To this end,
let $\left(\mu_n\right)_{n\in\mathbb{N}}$ satisfy Assumption~\ref{ass:1} and we assume $\supp(\mu)\subseteq \supp(\mu_n)$ for all $n\in\mathbb{N}$.
\begin{thm}\label{graph_norm_conv}
Let $b\in\{N,D\}$. For $f\in \mathcal{D}\left(\bar \Delta_{\mu}^b\right)$ there exists  $(f_n)_{n\in\mathbb{N}}$ with $f_n\in \mathcal{D}\left(\bar \Delta_{\mu_n}^b\right)$ such that for $n\in\mathbb{N}$
\begin{align*}
\lim_{n\to\infty} \left\lVert \pi_n f-f_n\right\rVert_{\infty}
+\left\lVert \pi_n \bar \Delta_{\mu}^b f-\bar \Delta_{\mu_n}^b f_n\right\rVert_{\infty} = 0.
\end{align*}
\end{thm}

\begin{proof}
Let $\lambda>0$, $f\in \mathcal{D}\left(\bar \Delta_{\mu}^b\right)$  and $g\coloneqq \left(\lambda - \bar \Delta_{\mu}^b\right)f$. Then, $f = \bar R_{\lambda}^b g$ and we define $f_n\coloneqq \bar R_{\lambda,n}^b \pi_n g$. Applying Theorem \ref{resolvent_conv},
\begin{align}
\lim_{n\to\infty}\left\lVert \pi_n f-f_n\right\rVert_{\infty} = 0. \label{graph_norm_proof}
\end{align}
Further, 
\begin{align*}
\bar \Delta_{\mu}^b f = \lambda f-\left(\lambda - \bar \Delta_{\mu}^b f\right)f = \lambda f - g
\end{align*}
and
\begin{align*}
\bar \Delta_{\mu_n}^b f_n = \lambda f_n-\left(\lambda - \bar \Delta_{\mu_n}^b\right)f_n = \lambda f_n - \pi_n g.
\end{align*}
It follows
\begin{align*}
\left\lVert \pi_n \bar \Delta_{\mu}^b f-\bar \Delta_{\mu_n}^b f_n\right\rVert_{\infty} = \lambda \left\lVert \pi_n f-f_n\right\rVert_{\infty}
\end{align*}
and thus, by \eqref{graph_norm_proof},
\begin{align*}
\lim_{n\to\infty}\left\lVert \pi_n \bar \Delta_{\mu}^b f-\bar \Delta_{\mu_n}^b f_n\right\rVert_{\infty} = 0.
\end{align*}
\end{proof}

\subsection{Strong Semigroup Convergence}\label{Strong Semigroup Convergence}

For $b\in\{N,D\}$ let $\left(T_t^b\right)_{t\geq 0}$ be defined by $\mu$, $\left( T_{t,n}^b\right)_{t\geq 0}$ be defined by $\mu_{n}$ and analogously 
the restricted semigroups $\left(\bar T_t^b\right)_{t\geq 0}$ and $\left(\bar T_{t,n}^b\right)_{t\geq 0}$ be defined by $\mu$ and $\mu_n$, respectively. 
The main result of this paper is a direct consequence of the previous results.

\begin{proof}[Proof of Theorem \ref{strong_semigroup_con_theorem}]
For $n\in\mathbb{N}$, $\pi_n$ is a bounded linear transformation between Banach spaces. Further, $\left(\bar T_t^b\right)_{t\geq 0}$ and $\left(\bar T_{t,n}^b\right)_{t\geq 0}, ~ n\in\mathbb{N}$ are strongly continuous contraction semigroups on their respective spaces (see Corollary \ref{semigroup_cor}). Hence, due to \cite[Theorem 6.1]{EKM}, the assertion is a direct consequence of Theorem \ref{graph_norm_conv}.
\end{proof}

Strong semigroup convergence can be interpreted as convergence of  solutions to heat equations. The connection is given as follows (see \cite[Proposition 6.2]{ENO}).

\begin{lem}
Let $A$ be the generator of a strongly continuous semigroup $\left(S_t\right)_{t\geq 0}$ on a Banach space $X$. Then, for each $f\in \mathcal{D}(A)$ the abstract heat equation
\begin{align}\label{classical_heat_equation}
\begin{split}
\frac{\partial u}{\partial t}(t) &= Au(t), ~ t\geq 0\\
u(0)&=f
\end{split}
\end{align}
has a unique classical solution on $X$ given by 
\begin{align*}
u: [0,\infty) \to X,~ t \mapsto S_tf,
\end{align*}
meaning that $u$ is continuously differentiable with respect to $X$, $u(t)\in\mathcal{D}\left(A\right)$ and \eqref{classical_heat_equation} holds for all $t\geq 0$.
\end{lem}

Let $T>0$ and $f\in\mathcal{D}\left(\bar\Delta_{\mu}^b\right)$. Theorem \ref{strong_semigroup_con_theorem} implies that the classical solution to
\begin{align*}
\frac{\partial u_n}{\partial t}(t) &= \bar\Delta_{\mu_n}^b u_n(t),\\
u_n(0)&=\pi_n f
\end{align*} 
converges uniformly for $(t,x)\in[0,T]\times[0,1]$ to the classical solution to
\begin{align*}
\frac{\partial u}{\partial t}(t) &= \bar\Delta_{\mu}^b u(t),\\
u(0)&=f
\end{align*}
as $n\to\infty$, assuming that $\pi_n f\in \mathcal{D}\left(\bar\Delta_{\mu_n}^b\right)$. However, the assumption $f\in\mathcal{D}\left(\bar\Delta_{\mu}^b\right)$ and $\pi_n f\in \mathcal{D}\left(\bar\Delta_{\mu_n}^b\right)$ for all $n\in\mathbb{N}$ is very restrictive, as the following example illustrates.
\begin{exa}
Let $\mu$ be a measure according to our conditions such that $\supp(\mu)$ is a $\lambda^1$-zero set and assume that $\supp(\mu_n)=[0,1]$ for all $n\in\mathbb{N}$. Further, let $f\in
\mathcal{D}\left(\bar\Delta_{\mu}^b\right)$. Then, on any interval $I\subseteq[0,1]\setminus \supp(\mu)$, $\pi_nf$ is linear. Now, if we assume that $\pi_nf\in\mathcal{D}\left(\bar\Delta_{\mu_n}^b\right)$, then $\bar\Delta_{\mu_n}^b f(x) = 0$, $x\in I$ and thus $\bar\Delta_{\mu_n}^bf=0\in \left(C[0,1]\right)^b_{\mu_n}$. If $b=D$, we obtain $\pi_nf=0\in \left(C[0,1]\right)^D_{\mu_n}$ and thus $f=0\in \left(C[0,1]\right)^b_{\mu}$ and if $b=N$, $(\pi_nf)'=0\in C[0,1]$ and thus $f'=0\in \left(C[0,1]\right)^N_{\mu}$.
\end{exa} 

This motivates the following solution concept (see \cite[Proposition 6.4]{ENO}).

\begin{defi}
Let $X$ be a Banach space, $A: X\to X$ and $f\in X$. We call a map $u: [0,\infty)\to X,~ t\mapsto u(t)$ a solution to 
the abstract heat equation
\begin{align}\label{mild_heat_equation}
\begin{split}
\frac{du}{d t}(t) &= Au(t), ~ t\geq 0,\\
u(0)&=f
\end{split}
\end{align}
if $u$ is continuous with respect to $X$ for $t\geq 0$, $u(t)\in \mathcal{D}(A)$ for all $t>0$ and $\lim_{h\to 0}\frac{u(t+h)-u(t)}{h}= Au(t)$ with respect to $X$ for $t>0$.
\end{defi}



Using this solution concept, we can establish the desired convergence for any initial condition in the appropriate space.

\begin{thm}\label{heat_theorem}
Let $f\in (C[0,1])_{\mu}^b$ and let $(\mu_n)_{n\in\mathbb{N}}$ satisfy Assumption~\ref{ass:1}. Further, let $\{u(t):t\geq 0\}$ be the unique solution to
\begin{align}\label{heat_final_1}
\begin{split}
\frac{d u}{d t}(t) &= \bar\Delta_{\mu}^b u(t), ~ t\geq 0,\\
u(0)&=f
\end{split}
\end{align}
and let for $n\geq 1$ $\{u_n(t):t\geq 0\}$ be the unique solution to
\begin{align}\label{heat_final_2}
\begin{split}
\frac{d u_n}{d t}(t) &= \bar\Delta_{\mu_n}^b u_n(t), ~ t\geq 0,\\
u_n(0)&=\pi_n f.
\end{split}
\end{align} 
Then,
\begin{align}\label{heat_final_3}
\lim_{n\to\infty} \left\lVert \pi_n u(t)- u_n(t)\right\rVert_{\infty}=0,
\end{align}
uniformly on bounded time intervals.
\end{thm}
\begin{proof}
First, we show that $t\mapsto \bar T_t^b f$ is a solution to \eqref{heat_final_1}. 
Let $t>0$. By \eqref{semigroup_range} we have for any $k\in\mathbb{N}$ \[u(t)= \bar T_t^b f = T_t^b f\in\mathcal{D}\left(\left(\Delta_{\mu}^b\right)^k\right).\] It follows that $\Delta_{\mu}^b u(t)\in\mathcal{D}\left(\Delta_{\mu}^b\right)$ and especially $\Delta_{\mu}^b u(t)\in (C[0,1])_{\mu}^b$, which implies $u(t)\in \mathcal{D}\left(\bar \Delta_{\mu}^b\right)$. 
From the strong continuity of $\left(\bar T_t^b\right)_{t\geq 0}$ along with the semigroup property we get the continuity of $u$ with respect to $(C[0,1])_{\mu}^b$. Further, since $\bar \Delta_{\mu}^b$ is the infinitesimal generator of $\left(\bar T_t^b\right)_{t\geq 0}$,
\begin{align*}
\lim_{h\to 0}\frac{u(t+h)-u(t)}{h}=\lim_{h\to 0}\frac{\bar T_h^b \bar T_t^bf-\bar T_t^b f}{h}=\bar \Delta_{\mu}^b\bar T_t^b f = \bar \Delta_{\mu}^b u(t).
\end{align*}
For the proof of uniqueness, first note that the unique solution to
\begin{align}\label{heat_final_proof}
\begin{split}
\frac{d v}{d t}(t) &= \Delta_{\mu}^b v(t), ~ t \geq 0\\
v(0)&=f
\end{split}
\end{align}
on the Hilbert space $L^2([0,1],\mu)$ is given by $v(t)=T_t^b f$ (see \cite[Theorem B.2.6]{KA}). We now show that a solution to \eqref{heat_final_1}, which we denote by $u$, is also a solution to \eqref{heat_final_proof}.  The continuity with respect to $L^2([0,1],\mu)$ follows from\begin{align*}
\left\lVert u(t)-u(s)\right\rVert_{\mu}\leq \left\lVert u(t)-u(s)\right\rVert_{\infty}, ~~ s,t \geq 0.
\end{align*} 
Let $t>0$. We have $u(t)\in \mathcal{D}\left(\bar \Delta_{\mu}^b\right)$, which by definition implies that $u(t) \in \mathcal{D}\left( \Delta_{\mu}^b\right)$. 
Further,
\begin{align*}
\lim_{h\to 0}\left\lVert \frac{u(t+h)-u(t)}{h} - \Delta_{\mu}^b u(t)\right\rVert_{\mu}
&=\lim_{h\to 0}\left\lVert \frac{u(t+h)-u(t)}{h} - \bar\Delta_{\mu}^b u(t)\right\rVert_{\mu}\\
&\leq \lim_{h\to 0}\left\lVert \frac{u(t+h)-u(t)}{h} - \bar\Delta_{\mu}^b u(t)\right\rVert_{\infty}\\&=0.
\end{align*}
Therefore, $u$ is a solution to  \eqref{heat_final_proof}. This proves the  uniqueness. We can follow the same arguments to verify that  $\bar T_{t,n}^b \pi_n f$ is the unique solution to \eqref{heat_final_2} for $n\in\mathbb{N}$. Then, \eqref{heat_final_3} is a direct consequence of Theorem \ref{strong_semigroup_con_theorem}.
\end{proof}

\section{Applications}\label{Application}

\begin{exa}\label{exa1} {\rm
As a first application, we consider a non-atomic Borel probability measure $\mu$ on $[0,1]$ such that $0,1\in\supp(\mu)$ and $\supp(\mu)\neq [0,1]$. We define for $\varepsilon\in(0,1)$ the approximating probability measure $\mu_{\varepsilon}$ by
\begin{align*}
\mu_{\varepsilon} \coloneqq \frac{\mu+\varepsilon\lambda^1}{1+\varepsilon}.
\end{align*}
It is elementary that $\mu_{\varepsilon}$ converges weakly to $\mu$ as $\varepsilon\to 0$ 
and Theorem \ref{heat_theorem} is applicable. Let $b\in\{N,D\}$ and $f\in (C[0,1])_{\mu}^b$. Then, the unique solution $\{u_{\varepsilon}(t):t\geq 0\}$  to
\begin{align*}
\begin{split}
\frac{d u_{\varepsilon}}{d t}(t) &= \bar\Delta_{\mu_{\varepsilon}}^b u_{\varepsilon}(t),\\
u_{\varepsilon}(0)&=\pi_{\varepsilon} f,
\end{split}
\end{align*} 
where $\pi_{\varepsilon}: (C[0,1])_{\mu}^b\to (C[0,1])_{\mu_{\varepsilon}}^b$ is an embedding as previously defined (see \eqref{embedding}),
converges to the unique solution $\{u(t):t\geq 0\}$ to
\begin{align*}
\begin{split}
\frac{d u}{d t}(t) &= \bar\Delta_{\mu}^b u(t),\\
u(0)&=f
\end{split}
\end{align*}
with respect to the uniform norm as $\varepsilon$ tends to zero.} 
\end{exa}
In the previous example, $\mu$ could be chosen to be an absolutely continuous measure, for example  $\lambda^1_{|_{\left[0,\frac{1}{3}\right]\cup\left[\frac{2}{3},1\right]}}$, or to be a singular measure, as a self-similar measure on the Cantor set. Furthermore, it is not required that the approximating measures have full support. 
\begin{figure}[t]
\centering
\includegraphics[scale=0.5]{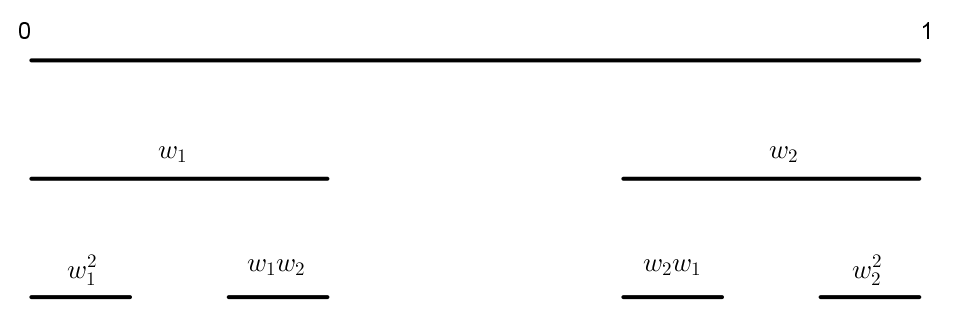}
\caption{Approximating Cantor measures of levels $n=0,1,2.$}
\label{cantor_measure}
\end{figure}

\begin{exa}\label{exa2}{\rm
Let $w_1,w_2\in(0,1)$ such that $w_1+w_2=1$ and let $\mu$ be the unique invariant Borel probabiliy measure on $[0,1]$ given by the IFS consisting of $S_1(x)=\frac{x}{3}$ and $S_2(x)=\frac{2}{3}+\frac{x}{3}, x\in[0,1]$ and weights $w_1$ and $w_2$, i.e. $\mu$ is a so-called Cantor measure. Following \cite{FME}, for $n\in\mathbb{N}$ we define the approximating Cantor measures of level $n$ by
\begin{align*}
\mu_n(B) \coloneqq 3^n\sum_{x\in\{1,2\}^n} \lambda^1_{|_{I_x}}\prod_{i=1}^n \omega_{x_i}, ~ B\in B([0,1]),
\end{align*}
where $I_x\coloneqq \left(S_{x_1}\circ...\circ S_{x_n}\right)([0,1]),~ x\in\{1,2\}^n$. The approximating Cantor measures of levels $n=0,1,2$ are illustrated in Figure \ref{cantor_measure}. We denote the distribution function of $\mu$ by $F$ and the distribution function of $\mu_n$ by $F_n$ for $n\in\mathbb{N}$. Then, 
$\left\lVert F-F_n\right\rVert_{\infty}\to 0$  (see \cite[Proposition 4.2]{FME}) as well as $\supp(\mu)\subset\supp(\mu_n)$ for $n\in\mathbb{N}$ and Theorem \ref{heat_theorem} can be applied. Hence,
for $f\in (C[0,1])_{\mu}^b$, the unique solution $\{u_n(t):t\geq 0\}$  to
\begin{align*}
\begin{split}
\frac{d u_n}{d t}(t) &= \bar\Delta_{\mu_n}^b u_n(t),\\
u_n(0)&=\pi_n f
\end{split}
\end{align*} 
converges to the unique solution $\{u(t):t\geq 0\}$ to
\begin{align*}
\begin{split}
\frac{d u}{d t}(t) &= \bar\Delta_{\mu}^b u(t),\\
u(0)&=f
\end{split}
\end{align*}
with respect to the uniform norm as $n$ tends to infinity.}
\end{exa}

Finally, we connect both applications. 

\begin{exa}{\rm
Let $\varepsilon>0,~ n\in\mathbb{N}$ and let $\mu$, $\mu_n$, $\{u(t):t\geq 0\}$ and $\{u_n(t):t\geq 0\}$ be defined as in Example \ref{exa2}. We define $\mu_{n,\varepsilon}$ by
\begin{align*}
\mu_{n,\varepsilon} \coloneqq \frac{\mu_n+\varepsilon\lambda^1}{1+\varepsilon},
\end{align*}
i.e. analogously to Example \ref{exa1}, and
$\{u_{n,\varepsilon}(t):t\geq 0\}$ to be the solution to 
 \begin{align*}
\begin{split}
\frac{d u_{n,\varepsilon}}{d t}(t) &= \bar\Delta_{\mu_{n,\varepsilon}}^b u_{n,\varepsilon}(t),\\
u_{n,\varepsilon}(0)&=\pi_{n,\varepsilon} f,
\end{split}
\end{align*}
 where $ \pi_{n,\varepsilon}$ is an embedding as previously defined.  Further, let  $t\in[0,\infty)$ and $\delta>0$.  
 By Example \ref{exa2}, there exists $n_0\in\mathbb{N}$ such that for all $n\geq n_0$ it holds
  \begin{align*}
\left\lVert u(t) - u_{n}(t)\right\rVert_{\infty} <\frac{\delta}{2}.
\end{align*}
 By Example \ref{exa1}, for each $n\geq n_0$ there exists $\varepsilon_n>0$ such that for all $\varepsilon<\varepsilon_n$ it holds
 \begin{align*}
\left\lVert u_n(t) - u_{n,\varepsilon}(t)\right\rVert_{\infty} <\frac{\delta}{2}.
\end{align*}
Hence,  for all $n\geq n_0,~ \varepsilon<\varepsilon_n$ it holds 
\begin{align*}
\left\lVert u(t) - u_{n,\varepsilon}(t)\right\rVert_{\infty} <\delta.
\end{align*}
Hence, the heat on a rod with mass distribution given by a Cantor measure diffuses approximately like the heat on a rod possessing a strictly positive mass density which is small off the Cantor set.}
\end{exa}

\section{Directions for Further Research}\label{Directions for Further Research}

\begin{bem}{\rm
Consider the heat equation \eqref{heat_equation_intro} with initial value given by the Delta distribution $\delta_y: g\mapsto g(y)$ for $y\in\supp(\mu)$. Then, the heat kernel
\begin{align*}
p_t(x,y)=\sum_{k\geq 1}e^{-\lambda_k^b}\varphi_k^b(x)\varphi_k^b(y), ~ (t,x)\in[0,\infty)\times[0,1]
\end{align*}
solves the equation in the distributional sense, where $\left\{\lambda_k^b, ~ k\geq 1\right\}$ are the ascending ordered eigenvalues and $\left\{\varphi_k^b,~ k\geq 1\right\}$ the $L_2([0,1],\mu)$-normed eigenfunctions of $\Delta_{\mu}^b$ on $L_2([0,1],\mu)$. The heat kernel is of particular importance in the context of the associated Markov process (see the remark below) and stochastic partial differential equations (see \cite{ES,ESW}). It is an open question whether weak measure convergence implies convergence of the corresponding heat kernels in an appropriate sense. 
}\end{bem}

\begin{bem}{\rm
The operator $\Delta_{\mu}^b$ on $L_2([0,1],\mu)$ is the infinitesimal generator of a Markov process, called a quasi-diffusion (see, e.g., \cite{IKD,KO,KS,LC}). Convergence of semigroups raises the question whether the associated Markov processes also converge weakly. If $\mu_n\rightharpoonup\mu$, our results imply that for each $f\in(C[0,1])_{\mu}^b$, $t\in[0,\infty)$ and each starting point $x\in[0,1]$
\begin{align*}
\mathbb{E}\left[f\left(X^b_n(t)\right)\right] = T_{t,n}^bf(x) \to T_{t}^b f (x)  = \mathbb{E}\left[f\left(X^b(t)\right)\right],~ n\to\infty, 
\end{align*}
where $X^b$ is associated to $\Delta_{\mu}^b$  and $X_n^b$ is associated to $\Delta_{\mu_n}^b$. This would need to be extended to a proof of 
convergence of all finite-dimensional distributions, and tightness would also be required, to establish that $X_n^b\to X^b$ weakly in the Skorokhod 
space of càdlàg functions.}
\end{bem}

\begin{bem}{\rm
Let $\mu$ be of full support. Consider the analgue of the wave equation
\begin{align*}
\begin{split}
\frac{d^2 u}{d t^2}(t) &= \Delta_{\mu}^b u(t), ~~ t\in[0,\infty)
\end{split}
\end{align*}
on $L^2([0,1],\mu)$. This hyperbolic equation describes the motion of a vibrating string with mass distribution $\mu$ such that, if it is deflected, a tension force drives it back towards its state of
equilibrium. If $\mu$ were not of full support, the string would have massless parts. It is not clear how to interpret massless parts of a string. We suppose that the motion of such a string behaves approximately like the motion of a string with very little mass on these gaps, analogous to our results about 
the diffusion of heat. 

Assume that $u(0)\in \mathcal{D}\left(\Delta_{\mu}^b\right)$ and, for reasons of simplicity, that the initial velocity vanishes. Then, there exists a 
unique solution on $L_2([0,1],\mu)$ given by $u(t)=C(t)u(0)$, $t\geq 0$, where $\left\{ C(t):t\geq 0\right\}$ denotes the strongly continuous cosine family 
of $\Delta_{\mu}^b$ (see, e.g., \cite{WA}). We have already shown that $\mu_n\rightharpoonup\mu$ implies strong resolvent convergence of the corresponding 
operators restricted to continuous functions. It is well-known that this implies convergence of the corresponding cosine families 
$\left\{ C_n(t):t\geq 0\right\}$, which implies convergence of the solutions to the corresponding wave equation, provided that there exists $M>0$ and $w\geq 0$ such that for all $n\geq 1$, $t\geq 0$ 
  $\left\lVert C_n(t)\right\rVert \leq Me^{w|t|}$ (see \cite{GO}). Proving that the restriction of $C(t)$ to 
$(C[0,1])_{\mu}^b$ is the cosine family of $\bar\Delta_{\mu}^b$ (and analogously for $\mu_n)$ and verifying the above estimate would be a way to establish 
the desired convergence of solutions to the wave equation.
}
\end{bem}

\end{document}